\documentclass[12pt]{amsart}
\usepackage{enumerate,amssymb}
\usepackage{amsmath}
\usepackage{comment}
\usepackage{hyperref}

\usepackage[top=1in, left=1in, right=1in, bottom=1in]{geometry}

\newtheorem{theorem}{Theorem}[section]
\newtheorem{lemma}[theorem]{Lemma}

\theoremstyle{definition}
\newtheorem{definition}[theorem]{Definition}
\theoremstyle{remark}
\newtheorem{remark}[theorem]{Remark}
\numberwithin{equation}{section}

\newcommand{\R}{{\mathbb R}}

\title[Boundary Harnack]{The Inhomogeneous Boundary Harnack Principle for Fully Nonlinear and $p$-Laplace equations}

\author{Mark Allen}
\address[Mark Allen]{Department of Mathematics, Brigham Young University, Provo,  UT}
\email{allen@mathematics.byu.edu}

\author{Dennis Kriventsov}
\address[Dennis Kriventsov]{Department of Mathematics, Rutgers University,  Piscataway, NJ}
\email{dnk34@math.rutgers.edu}

\author{Henrik Shahgholian}
\address[Henrik Shahgholian]{Department of Mathematics, KTH  Royal Institute of Technology, Stockholm, Sweden}
\email{henriksh@math.kth.se}

\date{October 22, 2020}

\begin{document}

\begin{abstract}
We prove a boundary Harnack principle in Lipschitz domains with small constant for fully nonlinear and $p$-Laplace type equations with a right hand side, as well as for the Laplace equation on nontangentially accessible domains under extra conditions. The approach is completely new and gives a systematic approach for proving similar results for a  variety of equations and geometries.
\end{abstract}

\maketitle

\section{Introduction}

This work is intended as a sequel to \cite{as19} by the first and third authors, where a boundary Harnack principle (BHP) was established for the Laplace equation with right hand side in Lipschitz domains (with small Lipschitz norm). Here we extend the result to the case of fully non-linear as well as $p$-Laplace equations. The novel and very simple approach introduced here also allows us to consider nontangentially accessible (NTA) domains when there is an assumed lower bound on the growth of the solution from the boundary.

In lay terms, the  main result in \cite{as19}  states that  (up to a multiplicative  constant) a positive harmonic function can dominate a superharmonic function close to a boundary point $x^0$ of  a domain\footnote{Throughout the paper we assume all domains are in $\R^n$ with $n \geq 2$.}
 $D \subset \R^n $ ($n \geq 2$), so long as both functions have zero boundary values in a small neighborhood of $x^0$. 
 See below Theorem \ref{t:flmain} for an exact formulation of the general case in this paper.

The  BHP with right hand side can be used to prove the regularity 
of free boundaries,  for the obstacle problem (see \cite{as19}), and the thin obstacle problem (see \cite{RT}).
Therefore  further study and generalization of the BHP with right hand side  should be emphasised to allow   applications to  more complicated free boundary problems. This work aims to make progress in this direction.

The reader may find it useful to read the longer introduction and applications mentioned in \cite{as19}, which we have chosen not to repeat here. Since then there has been some further research on this topic, including \cite{Sir}, as well as \cite{RT} which we learned about in the final stages of preparation of this work. The approach taken here is rather different and allows treatment of very general configurations (see  Theorem \ref{t:meta}  in Section \ref{sec:meta}); however, our results do not entirely overlap with the above mentioned references.

Our main results in this paper are the following theorems:
\begin{theorem}    \label{t:flmain}
 Let $\Omega$ be a Lipschitz domain with Lipschitz constant $L \leq \eta$ and assume $0 \in \partial \Omega$. Let $u,v \geq 0$ with $u=v=0$ on $\partial \Omega \cap B_1$ and assume  $u(e_n/2) = v(e_n/2) =1$. Assume the fully nonlinear operator $F$ satisfies the structural conditions \eqref{e:struc1} and \eqref{e:construc2}. There exist constants 
$C, \epsilon, \eta_0>0$ (depending on dimension and the ellipticity constants $\lambda, \Lambda$ of $F$) such that if 
 \[
  -1 \leq F(D^2 u, \nabla u), F(D^2 v, \nabla v) \leq \epsilon, 
 \]
 then  
  \[
   \frac{v}{u} \leq C \text{ in } \Omega \cap B_{1/2}. 
  \]
\end{theorem}

For the $p$-Laplacian we obtain a similar result for supersolutions:
\begin{theorem}    \label{t:flpmain}
 Let $\Omega$ be a Lipschitz domain with Lipschitz constant $L$ and assume $0 \in \partial \Omega$. Let $u,v \geq 0$ with $u=v=0$ on $\partial \Omega \cap B_1$ and assume  $u(e_n/2) = v(e_n/2) =1$. There exist constants 
$C, \eta>0$  such that if $L \leq \eta$ and 
 \[
  -1 \leq \Delta_p v, \Delta_p u  \leq 0, 
 \]
 then  
  \[
   \frac{v}{u} \leq C \text{ in } \Omega \cap B_{1/2}. 
  \]
\end{theorem}

The main ingredients in applying our method are 
\begin{description}
 \item[A] A boundary Harnack principle for solutions (to the homogeneous equation, with no right hand side).
 \item[B] An appropriate lower bound on the growth of $u,v$ from the boundary.
 \item[C] A comparison principle for sub and supersolutions.
 \item[D] Solvability of the Dirichlet problem (with continuous data).
\end{description}

Our theorems require a Lipschitz boundary for (A) and (B). However, for operators such as the Laplacian $\Delta$, one has a boundary Harnack principle for 
NTA domains \cite{JK}. Furthermore, for many free boundary problems a lower bound on the growth from the free boundary is often obtained directly (using competitors, barriers, or other techniques). In those cases (B) may be difficult or impossible to verify in general, but will be already available for the specific functions being considered. To handle this situation, one may apply our method on NTA domains and obtain the following conditional theorem, which appears useful in practice:

\begin{theorem}  \label{t:nta}
 Let $\Omega$ be an NTA domain with $0 \in \partial \Omega$, and assume that  for some $x^0 \in \Omega$ we have  $u(x^0)=v(x^0)=1$ and $u,v \geq 0$ with $u=v=0$ on $\partial \Omega \cap B_1$. If for some $0<\beta<2$ and some $c>0$
 one has 
 \[
  u(x), v(x) \geq c(\text{dist}(x,\partial \Omega))^{\beta}
 \]
 then 
 there exist a constant $C_0$ (depending on $\beta$, the NTA constants, and dist$(x^0, \partial B_1)$) such that if 
 \[
  -1 \leq \Delta u, \Delta v \leq 1, 
 \]
 then 
 \[
  \frac{v}{u} \leq C_0 \text{ in } \Omega \cap B_{1/2}. 
 \]
\end{theorem}
Unlike in Theorems \ref{t:flmain} and \ref{t:flpmain}, we do not assume that $\Delta u \leq \epsilon$ small here, nor that the domain is somehow flat: the growth bound is all that is required, though it does carry some indirect implications about the geometry of $\partial \Omega$ and $\Delta u$.


\section{A Metatheorem}\label{sec:meta}

Let $H_\Omega$ be a family of operators mapping $ C(\Omega)\times C(\partial \Omega) \to C(\bar{\Omega})$, and $Q$ a collection of open sets. The map $H$ should be thought of as a solution operator, mapping boundary data and right hand sides to solutions of an elliptic PDE. Fix a particular open set $U \in Q$. Let $V \subset C(\bar{U})$ consist of some subset of functions $u \geq 0$ with $H[f, u] = u$ on $U$ for some $f$ (generally this may be interpreted as positive functions with $f$ bounded by $1$, but only some specific properties below will be relevant; in fact neither $u \geq 0$ nor $H[f, u] = u$ are used explicitly in the proof below). Assume the following properties for  $H_\Omega$:

\begin{enumerate}
	\item[(P1)] Localization: For every $r > 0$ and $x \in \bar{U}$, there is a set $U_{x, r} \in Q$ such that $U_{x, r} \subset B_{2r}(x)$ and 
	$U \cap B_r(x) = U_{x, r} \cap B_r(x)$.
	\item[(P2)] Homogeneity: $H_\Omega[0, 0] = 0$ for every $\Omega \in Q$.
	\item[(P3)] Solvability: If $\Omega \in Q$, then $H_{\Omega}[f, g] = g$ on $\partial \Omega$ for any $g \in C(\partial \Omega)$.
	\item[(P4)] Extension: If $\Omega \subset \Omega' $ are in $Q$, then $H_{\Omega}[f_\Omega, H_{\Omega'}[f, g]|_{\partial \Omega} ] = H_{\Omega'}[f, g]$ on $\Omega$.
	\item[(P5)] Comparison:\footnote{It is possible to replace this assumption with a {\it homogeneous minimum principle}: if $g \geq 0$, then $H_{\Omega}[0, g] \geq 0$, though without the full comparison principle some of the remarks and typical applications will not follow. In cases where lower-order terms interfere with the comparison principle, it may still be possible to apply the results here by treating the lower-order terms as an inhomogeneity instead. For example, when studying $\Delta u = - \lambda u$ for $\lambda > 0$, our theorem will apply if one first shows $u$ is bounded, and then sets $\Delta u = f = - \lambda u$, with $f \in [- C, 0]$. Here $H_\Omega$ should be set to the solution to the Laplace equation, not to the eigenvalue problem.}
	 If $f_1 \geq f_2$ and $g_1 \leq g_2$, then $H_{\Omega}[f_1, g_1] \leq H_{\Omega }[f_2, g_2]$.
 	\item[(P6)] Approximation: For any set $\Omega = U_{x, r}$ from (P1) with $r \leq \frac{1}{4}$,  $x \in B_{1/2} \cap \partial U_{x, r}$, and $u\in V$, we have $|u - H_\Omega[0, u]|\leq C_1 r^\zeta$ for some $\zeta > 1$ on $\Omega$.
	\item[(P7)] Harnack: For any $u \in V$ and $B_{2r}(x) \subset U$    
	\[
	\sup_{B_r(x)} u \leq C_2 [\inf_{B_r(x)} u + 1].
	\]
	\item[(P8)] Boundary Harnack: For any $a \in \partial U$ and $\Omega = U_{a, r}$ from (P1), let $u_1, u_2$ satisfy $H_\Omega[0, u_1] = u_1$ and $H_\Omega[0, u_2] = u_2$. Assume, moreover, that $u_1, u_2 \geq 0$ on $\Omega$ and $u_1, u_2 = 0$ on $\partial U \cap B_r(a)$. Then
	\[
	\frac{u(x)}{v(x)} \leq (1 + C_3 \frac{|x - y|^\alpha}{r^\alpha})\frac{u(y)}{v(y)}
	\]
	for any $x, y \in U \cap B_{r/2}(a)$.
\end{enumerate}

In addition, we will use the following concept of $1$-sided NTA (or uniform) domain:

\begin{definition}
	A domain $\Omega \subset \R^n$ is a $1$-{\it sided NTA domain} (with constant $K$) if it satisfies the following two conditions:
	\begin{enumerate}
		\item[(D1)] For every $x \in \partial \Omega$ and $0 < r < \text{diam}(\Omega)$, there exists a ball $B_{r/K}(y) \subset \Omega \cap B_r(x)$.
		\item[(D2)] For every $x, y \in \Omega$, there is a curve $\gamma : [0, 1] \rightarrow \Omega$ with $\gamma(0) = x$, $\gamma(1) = y$, $l(\gamma([0, 1])) \leq K |x - y|$, and $\min\{ l(\gamma([0, t])), l(\gamma([t, 1]))  \} \leq  K d(\gamma(t), \partial \Omega)$ for all $ t \in [0, 1]$. Here $l$ denotes length.
	\end{enumerate}
\end{definition}

Our main theorem can now be phrased as follows:

\begin{theorem}\label{t:meta} Let $Q$, $H_\Omega$, and $U$ satisfy (P1-8), assume that $U$ is a $1$-sided NTA domain with constant $K$, and $0\in \partial U$. Then there is a constant $c_* = c_*(n, K)$ such that the following holds: let $u_1, u_2 \in V$ with  $u_i > 0$ on $U \cap B_1$, $u_i = 0$ on $\partial U \cap B_1$ ($i=1,2$), and  assume that for some $\beta \in (0, \zeta)$ $u_i$ satisfies the growth condition
	\begin{equation}\label{lower-bound}
	u_i(x) \geq C_4 d^\beta(x, \partial U \cap B_1) \qquad \forall x \in U.
	\end{equation}
	 In addition, assume that $u_1(x^0) = 1$ for some    $x^0 \in B_{c_*} (0)$ with
	 $   d(x^0, \partial U)  \geq c_*^2 $.
	  Then
	\[
	\frac{u_1}{u_2} \leq C_*
	\]
	on $B_{c_*} \cap U$. The constant $C_*$ depends only on $n, K, C_1, C_2, C_3, C_4, \zeta, \beta$ (where $C_1, C_2, C_3$ are the constants from (P6-8)).
\end{theorem}

\begin{remark}\label{r:nta}
	While $U$ being a $1$-sided NTA domain suffices for our argument, verifying property (P3), and possibly (P1), will often require making stronger assumptions. For the Laplace equation, two-sided NTA domains (where the complement of $U$ also satisfies (D1)) do have these properties, and in particular (P1) may be found in \cite{JK}. If one is working on Lipschitz graph domains $D_{L, r}$ as we do below, some of the details here can be simplified, and the $U_{x, r}$ can simply be chosen to be $U \cap B_{r}(x)$. It is worth noting, however, that (P8) does hold on $1$-sided NTA domains at least for the Laplace equation \cite{A2}, and this is roughly the most general class of domains on which it might be expected to hold \cite{A1}.
\end{remark}

\begin{remark}
	Although we  only assume that $u_1(x^0) = 1$ in the above theorem,  the lower  bound
\eqref{lower-bound} automatically 	 implies that 	$u_2(x^0) \geq c$, 
while an upper bound for $u_2$ is unnecessary. 
	An abstract argument shows that for $u_1$, \eqref{lower-bound} may be replaced with a growth condition on the corresponding homogeneous equation, up to increasing the radii slightly: first, if $u_i = H_{U}[f_i, u_1]$, let $w = H_{U_{0, 1}}[f_i^-, u_1]$ and use 
	 the comparison principle  to ensure  $u_1 \leq w$. Thus it suffices to prove the theorem for $w$. Then set $h = H[0, u_1]$: this has $h \leq w$, so a growth estimate for $h$ implies the same  for $w$. Growth estimates for solutions to the homogeneous equation are equivalent to one another, from (P8); therefore  in some cases (e.g. $f_2 \geq 0$) this estimate on $u_1$ may be redundant or easily obtainable. On the other hand, an inspection of the proof shows that if $f_2 = 0$ (and if $f_2 \leq 0$, after applying the comparison principle), then \eqref{lower-bound} for $u_2$ may be replaced with the condition $u_2(x^0) \geq 1$: the approximating function $v_2$ is equal to $u_2$, so \eqref{e:meta} below is automatic.
\end{remark}

\begin{proof}
	Let $\{r_k\}_{k = 0}^\infty$ be a decreasing sequence of numbers $r_k \leq \frac{1}{4}r_{k - 1}$ to be determined below, with $r_0 = \frac{c_*^2}{2}$, and 
	$$A_k = \{x \in U \cap B_{c_* + \frac{r_{k - 1}}{c_*}} : r_{k} \leq d(x, \partial U)  \leq r_{k - 1}  \}, \qquad k \geq 1,$$
	 with $A_0 = \{ x \in U \cap B_{2 c_*} : d(x, \partial U) \geq r_0 \}$; the constant $c_*$  will be chosen below in terms of $K$ and $n$ only.
	 Here $K$ is a constant which will be determined later, depending only on the NTA constant of $U$. Let   
	\[
	M_k = \sup_{A_k} \frac{u_1}{u_2};
	\]
	as $u, v$ are continuous and positive, we have $M_k < \infty$, for each $k$. Our main goal is to estimate $M_k$ in terms of $M_{k - 1}$, but we first consider $M_0$.
	
	 Applying (D2), if $c_*$ is sufficiently small in terms of the NTA constant $K$, any $x, y \in B_{c_*} \cap U$ may be connected by a curve as described there which is contained in $B_{1/2}$. Furthermore, from (D1) for any $x \in \partial U$, and every $r < 1$, there is a ball $B_{r/K}(y) \subset U \cap B_r(x)$; so long as $c_* \leq \frac{1}{2K}$, $B_{r c_*/2}(y)$ has the same property. We now fix $c_*$ so that these properties hold.
	
	To estimate $M_0$, we first observe that by the lower bound assumption \eqref{lower-bound}
we have   $u_2 \geq C$ on $A_0$. 	 On the other hand, we know that $u_1(x^0) = 1$, that $x^0 \in A_0$ and any other point $x \in A_0$ may be connected to $x^0$ via a path in $B_{1/2}$ of bounded length and staying a  distance at least $c = r_0/K$ 
from the boundary $\partial U$. This path may be covered by finitely many balls of radius $c/2$, and applying the Harnack principle (P7) to each ball consecutively gives that $u_1(x)$ is bounded. Taking the supremum, we see that $M_0$ is bounded in terms of $K$ and the constant $C_2$ in (P7).
	
	Now take any point $x \in A_k$, and let  $y \in \partial U$ with $|y - x| \leq r_{k - 1}$. Use the NTA property to find a ball $B_{r_{k - 1}}(z) \subset U \cap B_{r_{k - 1}/2c_*}(y)$; then we have that $d(z, \partial U) \geq r_{k -1}$, while    
	\[
	|z| \leq |z - y| + |y - x| + |x|\leq  \frac{r_{k - 1}}{2 c_*}  + r_{k - 1} + c_* + \frac{r_{k - 1}}{c_*} \leq c_* + \frac{r_{k - 2}}{c_*} 
	\]
	if $k \geq 2$, using here that $r_{k - 1} \leq \frac{1}{4}r_{k - 2}$. If $k = 1$, then using $r_0 = c_*^2/2$ gives 
	\[
	|z| \leq \frac{r_{0}}{2 c_*}  + r_{0} + c_* + \frac{r_{0}}{c_*}  (\frac{1}{4} + \frac{c_*}{2} + 1 + \frac{1}{2})c_*\leq 2 c_*
	\]
	instead.	
	 Consider the line segment connecting the points $z$ and $y$: all points on this line segment must lie inside $B_{c_* + \frac{r_{k - 2}}{c_*}}$ and $B_{r_{k - 1}/2c_*}(y)$ as well, as both endpoints do and balls are convex. As $d(y, \partial U) = 0$, $d(z, \partial U) \geq r_{k-1}$, and the distance is continuous, we may find some $z^1$ on the line segment such that $d(z^1, \partial U) = r_{k - 1}$. In particular, the two important properties are that $z^1 \in A_{k - 1}$, while $x, z^1 \in B_{r_{k - 1}/2c_*}(y)$.
	
	Next we fix $U_{y, s}$ with $s \geq \frac{r_{k - 1}}{c_*}$ to be chosen below, and use (P1) and (P3) to find $v_1, v_2$ which satisfy $H_{U_{y, s}}[0, u_i] = v_i$ (recall that this is analogous to solving the homogeneous equation on $U_{y, s}$ with boundary data given by $u_i$). Note that by comparison (P5), with the function $0$, and homogeneity (P2), we have $v_i \geq 0$ on $U_{y, s}$
	There are two main estimates we need for $u_i$ and $v_i$. The first is from (P6): we have that $|v_i - u_i|\leq C_1 s^\zeta$ on $U \cap B_{s}(y)$. We may further combine it with the assumed growth estimate \eqref{lower-bound} to arrive at  
	\begin{equation}\label{e:meta}
	|u_i - v_i| \leq C_1 s^\zeta \leq u_i \frac{C_1 s^\zeta}{C_4 r_k^\beta}
	\end{equation}
	on $B_s(y) \cap (A_k \cup A_{k-1})$, that may be rephrased  as    	  	
	\begin{equation}
	(1 - C \frac{s^\zeta}{r_k^\beta}) u_i \leq	v_i \leq (1 + C \frac{s^\zeta}{r_k^\beta}) u_i,
	\end{equation}
	so after dividing
	\begin{equation}
	(1 - C \frac{s^\zeta}{r_k^\beta}) v_i \leq	u_i \leq (1 + C \frac{s^\zeta}{r_k^\beta}) v_i
	\end{equation}
	on this region, so long as $r_k^\beta$ is much larger than $s^\zeta$, which we will ensure below.
	
	On the other hand, we may apply (P8), the homogeneous boundary Harnack principle, to $v_i$. We apply it specifically with $r = s$, $a = y$, $x = x$, and $y = z^1$, to get that
		\begin{equation}\label{e:v1v2}
	v_1(x) \leq v_2(x) \frac{v_1(z^1)}{v_2(z^1)} (1 + C \frac{r_{k-1}^\alpha}{s^\alpha} ).
	\end{equation}	
	Now, $z^1 \in A_{k -1}$, so there we may argue as follows, using \eqref{e:meta}: 
		\[
	\frac{v_1(z^1)}{v_2(z^1)}  \leq 
 (1 + C \frac{s^\zeta}{r_k^\beta})^2 \frac{u_1(z^1)}{u_2(z^1)} 
  \leq (1 + C \frac{s^\zeta}{r_k^\beta}) M_{k - 1}.
	\]
		This along with \eqref{e:v1v2} gives
	\[
	\frac{v_1(x)}{v_2(x)} \leq  (1 + C \frac{s^\zeta}{r_k^\beta} + C \frac{r_{k-1}^\alpha}{s^\alpha} ) M_{k - 1},
		\]
	and finally using \eqref{e:meta} again but this time at $x$,
	\[
	\frac{u_1(x)}{u_2(x)} \leq	(1 + C \frac{s^\zeta}{r_k^\beta})^2 \frac{v_1(x)}{v_2(x)} \leq (1 + C \frac{s^\zeta}{r_k^\beta} + C \frac{r_{k-1}^\alpha}{s^\alpha} ) M_{k - 1}.
	\]
	This entire construction can be done at any $x \in A_k$, so taking the supremum gives
	\[
	M_k \leq (1 + C \frac{s^\zeta}{r_k^\beta} + C \frac{r_{k-1}^\alpha}{s^\alpha} ) M_{k - 1}.
	\]
	
	Now we must choose $s$ and $r_k$ in an appropriate manner; we have already required that $s^\zeta \ll r_k^\beta$, $r_{k - 1} \ll s$, and $r_k \leq r_{k -1}/4$. To proceed, select a $\gamma > 1$ such that $\beta \gamma < \zeta$, and set $r_k = r_{k - 1}^\gamma$. This immediately implies  that $r_k \leq \frac{r_{k-1}}{4}$. Next, choose a $\sigma < 1$ with $\zeta \sigma > \beta \gamma$, and set $s = r_{k-1}^\sigma = r_k^{\sigma/\gamma}$; this has the other two necessary properties. With these choices, our recurrence relation may be rewritten as
	\[
	M_k \leq (1 + C r_k^{\zeta \sigma / \gamma - \beta} + C r_k^{(1 - \sigma)\gamma\alpha} ) M_{k - 1}.
	\]
	As $r_k \leq \frac{1}{4}r_{k-1}$, $r_k \leq r_0 4^{-k}$, we'll have 
	\[
	M_k \leq (1 + C 4^{-ck})M_{k-1} \leq \prod_{i = 1}^{\infty}(1 + C 4^{-ck})M_0.
	\]
	This infinite product is finite, giving $M_k \leq C M_0$ for all $k$. As the union of the $A_k$ exhausts $B_{c_*} \cap U$, we have shown that
	\[
	\sup_{B_{c_*} \cap U} \frac{u_1}{u_2} \leq C M_0.
	\]
	This completes the proof.
\end{proof}

\begin{proof}[Proof of Theorem \ref{t:nta}]
Set $U = \Omega$, $Q$ the collection of (two-sided, as in Remark \ref{r:nta}) $NTA$ domains with constant at most $K$, and $H_{U_{x, r}}[f, g]$ the Perron solution to the Laplace equation on $U_{x, r}$. Set $V = \{ u \in C(\bar{U}) : u \geq 0, |\Delta u|\leq A  \text{ on } U \}$, with $A$ to be determined in terms of $c_*$ and the given constants only. Then (P1), (P3), and (P8) follow from \cite{JK} as long as $K$ is taken to be a sufficiently large multiple of the NTA constant of $\Omega$, while (P2), (P4), (P5), and (P7) are classical. The approximation property (P6) follows from an elementary barrier argument (as in Lemma \ref{l:remainder} below). After applying the Harnack inequality repeatedly, we have that
\[
C \geq u, v \geq c
\]
on $U \cap B_2 \cap \{ d(x, \partial U) \geq c^2_* \}$. Fix $A$ so that if we define the functions $u_1 = u(\cdot)/u(x^1)$, $u_2 = v(\cdot)/v(x^1)$ on $B_1(x)$ for $ x\in \partial U \cap B_1$ and $x^1$ a point in $B_{c_*}(x)$ a distance at least $c_*^2$ from $\partial U$, they have $|\Delta u_i|\leq \max\{\frac{1}{u(x^1)}\frac{1}{v(x^1)}\} \leq A$.

Applying Theorem \ref{t:meta} to $u_1$, $u_2$ on $B_1(x)$ for every $x\in \partial U \cap B_1$ gives that
\[
	\sup_{B_1 \cap U \cap \{z: d(z, \partial U) < c_* \}} \frac{u}{v} \leq C,
\]
which implies the conclusion.
\end{proof}



\section{Fully Nonlinear Equations}  \label{s:fne}

Let $S(n)$ be the set of symmetric $n \times n$ matrices, $\Lambda \geq \lambda>0$ and $M\geq 0$ be constants, and $P_{\Lambda,\lambda}^-, P_{\Lambda, \lambda}^+$ the extremal Pucci operators defined by 
\[
P_{\Lambda,\lambda}^-(R) = \lambda \sum_{e_i>0} e_i + \Lambda \sum_{e_i <0} e_i \qquad P_{\Lambda,\lambda}^+(R) = \Lambda \sum_{e_i>0} e_i + \lambda \sum_{e_i <0} e_i,
\]
where $e_i$ are the eigenvalues of $R$.

As our method requires that the boundary Harnack principle already holds for solutions to the homogeneous equation, we will require the same structural conditions for fully nonlinear 
equations as required in \cite{f01} where a boundary Harnack principle without right hand side is shown.  We therefore assume that $F: S(n) \times \R^n \rightarrow \R$ (the nonlinear operator in our equation $F(D^2 u, \nabla u) = f$) satisfies 
\begin{equation}   \label{e:struc1}
P_{\Lambda,\lambda}^- (R-S) - M|p-q| \leq F(R,p) - F(S,q) \leq P_{\Lambda, \lambda}^+(R-S) + M|p-q|
\end{equation}
for $R,S \in S(n)$ and $p,q \in \mathbb{R}^n$.

We also assume that $F$ is positively homogeneous of degree $1$, i.e.
\begin{equation}   \label{e:construc2}
 F(\gamma R, \gamma p) = \gamma F(R,p)  \quad \text{ for all } \gamma >0, \quad R \in S(n), p \in \mathbb{R}^n. 
\end{equation}

We follow \cite{f01, CIL} when we write $F(D^2 u, \nabla u) \leq (\geq) f$ {\it in the viscosity sense} for a continuous function $f$. The key property of viscosity solutions is the following comparison-type fact: if $F(D^2 u, \nabla u) \geq f$ and $F(D^2 v, \nabla v) \leq g$ on a domain $\Omega$, in the viscosity sense, then $P_{\Lambda, \lambda}^-(D^2 (v - u)) - M |\nabla v - \nabla u| \leq g - f$ in the viscosity sense on $\Omega$. The proof is straightforward if one of $v, u$ is $C^2$ from the definitions and \eqref{e:struc1}, but the general case may be derived from \cite{CIL}.

We recall the following notation from \cite{as19} for Lipschitz domains. We consider Lipschitz domains $D_{L,R}$ where 
\[
 D_{L,R} := \{(x',x_n) \in B_R : x_n > g(x')\},
\]
and $g$ is a Lipschitz function with constant at most $L$, that is $|g(x')-g(y')|\leq L|x'-y'|$. We will assume $g(0)=0$, and will write $D_{L,\infty}$ if $R=\infty$. 

\subsection{Approximation and homogeneous boundary Harnack}

We need the following classical boundary Harnack principle, which is Lemma 2.4 in \cite{f01}:
\begin{lemma}  \label{l:classic}
 Let  $F(D^2u, \nabla u)=F(D^2 v, \nabla v)=0$  in $D_{L,1}$ (in the viscosity sense), with $u,v\geq 0$, and $u=v=0$ on $\partial D_{L,1} \cap B_{3/4}$. 
  If $u(e_n/2)=v(e_n/2)=1$, then there exists $C(\Lambda, \lambda, M, L, n)$ such that 
 \[
  \frac{u}{v} \leq C \in D_{L,1/2}. 
 \]
  
\end{lemma}  

In our situation, it will be more convenient to apply the following slight variation of Lemma \eqref{l:classic}.
\begin{lemma}  \label{l:classic2}
 Let  $F(D^2u, \nabla u)=F(D^2 v, \nabla v)=0$ (in the viscosity sense)  in $D_{L,1}$, with $u,v\geq 0$, and $u=v=0$ on $\partial D_{L,1} \cap B_{3/4}$.
 If $u(x)=v(x)=1$ for some $x \in D_{L,1/2}$, then $u/v \leq C$ in $D_{L,1/2}$. 
\end{lemma} 

\begin{proof}
 We apply Lemma \ref{l:classic} to $\tilde{v}, \tilde{u}$ in place of $u, v$, where
 \[
  \tilde{u}:= \frac{u(x)}{u(e_n/2)} \quad \text{ and } \quad \tilde{v}:= \frac{v(x)}{v(e_n/2)}, 
 \]
and obtain 
\[
  C \geq \frac{\tilde{v}(x)}{\tilde{u}(x)}   = \frac{v(x)}{u(x)} \frac{u(e_n/2)}{v(e_n/2)} = \frac{u(e_n/2)}{v(e_n/2)}. 
\]
Now apply Lemma \ref{l:classic} again to $\tilde{u}, \tilde{v}$ in the opposite order to get that for any $y \in D_{L, 1/2}$,
\[
C \geq \frac{\tilde{u}(y)}{\tilde{v}(y)}   = \frac{u(y)}{v(y)} \frac{v(e_n/2)}{u(e_n/2)} \geq \frac{u(y)}{v(y)} \frac{1}{C}.
\]
This concludes the proof.
\end{proof}

We will also need the following lemma.
\begin{lemma}  \label{l:remainder}
 Let $v$ satisfy $-1 \leq F(D^2 v, \nabla v) \leq 1$ in $D_{L,R}$ (in the viscosity sense), where $R \leq 1$. If $v=h+w$ where $w$ solves 
 \[
  \begin{cases}
   F(D^2 w, \nabla w)=0 &\text{ in } D_{L,R} \\
   w=v &\text{ on } \partial D_{L,R},
  \end{cases}
 \]
 then there exists a constant $C = C(n, \lambda, \Lambda, M)$ such that 
 \[
  |h| \leq CR^2.
 \]
\end{lemma}

\begin{proof} 
  From earlier remarks we have that $h = v - w$ satisfies 
  \begin{equation} \label{e:comparison}
	  P^-_{\Lambda, \lambda}(D^2 h) - M |\nabla h| \leq 1 \qquad -1 \leq P^+_{\Lambda, \lambda}(D^2 h) + M |\nabla h|
  \end{equation}
  in the viscosity sense. 
  
  Assume first that $R \leq \frac{\lambda n}{2 M}$: we use 
 \[
 G(x):= \frac{ R^2}{\lambda n} \left(1 - \frac{|x|^2}{R^2}\right)
 \] 
  as an explicit barrier on $B_{R}$.  We have that 
 \[
  P^+_{\Lambda, \lambda}(D^2 G) + M |\nabla G| =  - 2 + 2 \frac{M R}{\lambda n} \leq -1.  
 \] 
 Since $h=0$ on $\partial D_{L,R}$ and $G \geq 0$ there, using the comparison principle with the right inequality in \eqref{e:comparison} we obtain that $h \leq G \leq C R^2$. 
 
 On the other hand, if $R \geq \frac{\lambda n}{2 M} \geq c(n, \lambda, M)$, we may use a barrier of the form
 \[
	 G(x) =  e^{2 S} - e^{S (x_1 + 1)}
 \]
 where $S \geq 0$. Then $G \geq 0$ on $D_{L, 1}$, and
 \[
	 P^+_{\Lambda, \lambda}(D^2 G) + M |\nabla G| = A e^{S (x_1 + 1)}[- S^2 \lambda + S M] \leq S e^{2 S}[ - S \lambda + M ] \leq -1
 \]
 if we select $ S = \max\{\frac{2 M}{\lambda}, 1\}$.
 This gives $h \leq G \leq  e^{2S} \leq C(n, \lambda, M) \leq C R^2$ on $D_{L, R}$ after applying the comparison principle.
 
 The opposite inequality follows by considering $-h$ instead, using the other viscosity inequality.
\end{proof}

\subsection{Growth estimates}

\begin{lemma} \label{l:barrier} There is a number $\epsilon = \epsilon(n, \lambda, \Lambda, M) > 0$ such that for every $\eta < \eta_0(n, \lambda, \Lambda, M)$, if $u \geq 0 $ on $D_{L, 1}$, $u(e_n/2) \geq 1$, $u = 0$ on $\partial D_{L, 1}\cap B_1$,
	\[
		- 1 \leq F(D^2 u, \nabla u) \leq \epsilon
	\]
	in the viscosity sense, and $L \leq \eta$, then
	\[
		u(x) \geq c_* (x_n - \eta)
	\]
	on $D_{L, 1/16}$ for a $c_* = c_*(n, \lambda, \Lambda, M)$ (which does not depend on $\eta$).
\end{lemma}

\begin{proof}
	We will show this using an explicit estimate with a barrier. The barrier argument proceeds in two steps, but they use the same function.
	
	Set $\phi(x) = |x|^{-q}$. 
		 Direct computation shows that for $q$ sufficiently large in terms of $M$, $\lambda$, and $\Lambda$, we have that
	\[
		P^{-}_{\lambda, \Lambda}(D^2 \phi) - M |\nabla \phi| \geq 1
	\]
	for $|x|\leq 1$. Fix $q$ to be such a value, this implies that $F(D^2 \phi, \nabla \phi) \geq 1$.
	
	From the Krylov-Safonov Harnack inequality, we have that
	\[
		\inf_{B_{\kappa}(e_n/2)} u \geq c \sup_{B_{\kappa}(e_n/2)} u - C \kappa^2 \geq c
	\]
	for a small $\kappa \ll \frac{3}{8}$, $c$ depending only on the ellipticity constants. Consider the barrier function
	\[
		h(x) = c \frac{\phi(x - e_n/2) - (3/8)^{-q}}{\kappa^{-q} - (3/8)^{-q}}
	\]
	defined on the annulus $A = B_{3/8}(e_n/2) \setminus B_{\kappa}(e_n/2)$. On the outer boundary, we have that $h = 0$, while on the inner boundary, $h = c \leq u$. On $A$, we have
	\[
		F(D^2 h, \nabla h) \geq P^{-}_{\lambda, \Lambda}(D^2 h) - M |\nabla h| \geq \frac{c}{\kappa^{-q} - (3/8)^{-q}} := \epsilon_1(n, \lambda, \Lambda, M).
	\]
	So long as $\epsilon < \epsilon_1$ (and $\eta_0$ is small enough so that $A \subset D_{L, 1}$), we may apply the comparison principle to $h$ and $u$ to obtain $u \geq h$ on $A$. In particular, take any point $y$ a distance at most $\frac{1}{32}$ from a point $z$ on the region $D = \{(x', x_n): |x_n - \frac{1}{2}| \leq \frac{1}{4}, |x'| < \frac{1}{10}  \}$; then
	\[
		|y - e_n/2| \leq |y - z| + |z - e_n/2| \leq \frac{1}{32} + \sqrt{(\frac{1}{4})^2 + (\frac{1}{16})^2} \leq \frac{11}{32} < \frac{3}{8}.
	\]
	Hence $u \geq h \geq c_0$ at any such point. To summarize: at any $x \in D$, $u \geq c_0$ on $B_{1/32}(x)$.
	
	Now we apply a similar argument around any $x \in D$ with $x_n = \frac{1}{4}$, except slightly more carefully. Fix $\eta$, and note that $D_{L, 1/16}$ contains the large region $B_{1/16} \cap \{x_n \geq \eta \}$. Define $ r = \frac{1}{4} - \eta$, the annulus $A_x = B_{r}(x)\setminus B_{\frac{1}{32}(x)}$ contained within this region, and the barrier function
	\[
		h_x(y) = c_0 \frac{\phi(y - x) - (r)^{-q}}{(1/32)^{-q} - (r)^{-q}}
	\]
	defined on $A_x$. As before, $h_x = 0$ on the outer boundary, $h_x = c_0 \leq u$ on the inner boundary, and
	\[
		F(D^2 h_x, \nabla h_x)\geq \frac{c_0}{(1/32)^{-q} - (r)^{-q}} \geq \frac{c_0}{(1/32)^{-q}}  := \epsilon_2(n, \lambda, \Lambda, M).
	\]
	If $\epsilon < \epsilon_2$ (which does not depend on $\eta$), we have that $u \geq h_x$ on $A_x$. Using the explicit form of $h_x$,
	\[
		u(x', t) \geq h_x(x', t) \geq c_1 [r - (x_n - t)] = c_1[t - \eta]
	\]
	for $t \in (\eta, \frac{1}{4} - \frac{1}{32})$, where $c_1$ can be taken independent of $\eta$.
	
	So far, we have shown that for any $x'$ with $|x'|\leq \frac{1}{16}$ and any $ t \in (\eta, \frac{1}{4} - \frac{1}{32})$,
	\[
		u(x', t) \geq c_1[t - \eta].
	\]
	For $t \leq \eta$, this inequality remains true automatically. In particular, this means that it  holds for all $(x', t) \in D_{L, 1/16}$, giving the conclusion.
\end{proof}

\begin{lemma}\label{l:growthnonlinear} Let $\beta \in (1, 2)$. Then there are constants $\epsilon, \eta > 0$ such that if $u \geq 0 $ on $D_{L, 1}$, $u(e_n/2) \geq 1$, $u = 0$ on $\partial D_{L, 1}\cap B_1$,
	\[
	- 1 \leq F(D^2 u, \nabla u) \leq \epsilon
	\]
	in the viscosity sense, and $L \leq \eta$, then
	\[
		u(x) \geq c_1 d^\beta(x, \partial D_{L, 1})
	\]
	for $x\in D_{L, 1/64}$.
\end{lemma}

\begin{proof}
	We begin by ensuring that $\eta, \epsilon$ are small enough and applying Lemma \ref{l:barrier} to learn that
	\[
		u(x) \geq c_* [x_n - \eta] \geq \frac{c_*}{2} x_n
	\]
	so long as $x_n \geq 2 \eta$ and $x \in D_{L, 1/16}$. Fix a point $x \in \partial D_{L, 1} \cap B_{1/64}$, and define
	\[
		u_1(y) = \frac{u(x + r_1 y)}{u(x + r_1/2 e_n)}
	\]
	where $r_1 = 64 \eta$. We claim that on $B_1$, $u_1$ satisfies all of the assumptions of Lemma \ref{l:barrier} (using $D'_{L, 1} = (D_{L, 1} - x)/r_1$). Indeed, most of the assumptions follow immediately: $u_1 \geq 0$ on $D'_{L, 1}$ and vanishes on the graphical boundary, has $u_1(e_n/2) = 1$ by construction, and $D'_{L, 1}$ has Lipschitz constant bounded by $L \leq \eta$. The main assumption we must check is that it satisfies the relevant differential inequalities. For this, rescaling gives that (for a $\tilde{F}$ which satisfies the same properties as $F$)
	\begin{equation}\label{e:growth}
	- \frac{r_1^2}{u(x + r_1/2 e_n)} \leq	\tilde{F}(D^2 u_1, \nabla u_1) \leq \frac{r_1^2}{u(x + r_1/2 e_n)} \epsilon,
	\end{equation}
	so we must show that $u(x + r_1/2 e_n) \geq r_1^2$. Note that $x + r_1/2 e_n$ has $n$-th component larger than $2 \eta$, so
	\[
		u(x + r_1/2 e_n) \geq \frac{c_*}{2} (x_n + r_1/2) \geq \frac{c_*}{8} r_1.
	\]
	So long as $c_*/8 > 64 \eta$ this is larger than $r_1^2$, so we may proceed so long as $\eta$ is chosen small enough.
	
	Applying Lemma \ref{l:barrier} to $u_1$ gives that
	\[
		u_1(y) \geq c_*/2 y_n
	\]
	for $y_n \geq 2 \eta$, which translates to
	\[
		u(x', t) \geq c_*/2 (t - x_n) \frac{u(x + r_1/2 e_n)}{r_1} \geq c_*^2/16 (t - x_n)
	\]
	for $t \in (x_n + 2 \eta r_1, x_n + r_1)$. Note that from our choices, $x_n + r_1 \geq 2 \eta$.
	
	We may continue to apply Lemma \ref{l:barrier} to $u_k$ around $x$, with $r_k = 64 \eta r_{k - 1}$, in a similar manner, and we claim that this gives
	\[
		u(x', t) \geq (\frac{c_*}{8})^k \frac{c_*}{2} (t - x_n)
	\]
	on $t \in (x_n + 2 \eta r_k, x_n + r_k)$. Let us verify this claim by induction on $k$, with the $k = 1$ case already complete. As before we must ensure that $u_k$ satisfies the differential inequalities, which follows from
	\[
		u(x + r_k/2 e_n) \geq (\frac{c_*}{8})^{k - 1} \frac{c_*}{2} \frac{r_k}{2} \geq (64 \eta)^k r_k \geq r_k^2; 
	\]
	we are using here that $r_k/2  \in (2 \eta r_{k - 1}, r_{k - 1})$ by construction and the inductive hypothesis. After applying the lemma and scaling back, we obtain for $t \in (x_n + 2 \eta r_k, x_n + r_k)$ that
	\[
		u(x', t) \geq c_*/2 (t - x_n) \frac{u(x + r_k/2 e_n)}{r_k} \geq (\frac{c_*}{8})^{k} \frac{c_*}{2}(t - x_n),
	\]
	as claimed. 
	
	Now fix $t \in (x_n, 2\eta)$, and find the largest $k$ for which $t \in (x_n + 2 \eta r_k, x_n + r_k)$. As these intervals cover $(x_n, 2\eta)$ this is well-defined, and as $r_k = (64 \eta)^k$, we have that
	\[
		k \leq \frac{\log (t - x_n)}{\log 64 \eta} \leq k + 1.
	\]
	Using this with our estimate on $u$,
	\[
		u(x', t) \geq (\frac{c_*}{8})^{k} \frac{c_*}{2}(t - x_n) \geq c (t - x_n)^{- c' \log \eta} (t - x_n) \geq c (t - x_n)^{1 - c' \log \eta}
	\]
	where $c, c'$ only depend on $c_*$ and explicit numbers. Select $\eta$ small enough that $1 - c' \log \eta < \beta$.
	
	Finally, we observe that if $ t \geq 2 \eta$, then from our very first estimate
	\[
		u(x', t) \geq \frac{c_*}{2} t \geq \frac{c_*}{4}(t - x_n),
	\]
	as $|x_n|\leq \eta$ from the Lipschitz nature of $D_{L, 1}$. For any point $z = (x', t)$ as above, $d(z, \partial D_{L, 1}) \leq (t - x_n)$ (as $(x', x_n)\in \partial D_{L, 1}$), so this reads
	\[
		u(x', t) \geq c d^{\beta}((x', t), \partial D_{L, 1})
	\]
	for $(x', t) \in D_{L, 1/64}$.
\end{proof}

\begin{proof}[Proof of Theorem \ref{t:flmain}]
	We write $\Omega = D_{L, 1}$, as above. Set $U = D_{L, 1}$, $Q$ to be the collection of all sets of the form $U \cap B_r(x)$, and $H_{U \cap B_r(x)}[f, g]$ the Perron solution operator, mapping right hand side $f$ and boundary data $g$ to the unique viscosity solution. Then properties (P1), (P2), and (P4) are immediate, while (P3) and (P5) follow from the viscosity comparison principle (see \cite{f01}). 
	Set $V = \{u \in C(\bar{U}): u \geq 0,  -A \leq  F(D^2 u, \nabla u)  \leq A \text{ on } U\}$ 
	with constant $A$ to be chosen; then (P6) follows from Lemma \ref{l:remainder}. Finally, (P7) is the Krylov-Safonov Harnack inequality, while (P8) may be found in \cite{f01}, Lemma 2.4, combined with Lemma \ref{l:classic2} above.
	
	Fix $\beta \in (1, 2)$ and apply Lemma \ref{l:growthnonlinear} to $u$ and $v$, selecting $\eta, \epsilon$ sufficiently small. This gives that
	\[
	u(x), v(x) \geq c d^\beta(x, \partial U \cap B_1)
	\]
	on $B_{1/2} \cap U$. This gives that if $x^0 = c_* e_n/4$, $u(x^0), v(x^0) \geq c$. On the other hand, applying the Harnack inequality on a region bounded away from $\partial U$ and containing $e_n/2, x^0$ gives that $u(x^0), v(x^0) \leq C$. The functions $u_1(x) = u(x)/u(x^0)$ and $u_2(x) = v(x)/v(x^0)$ solve $- C \leq F(D^2 u_i,  \nabla u_i)  \leq C$, so in particular in $V$ if $A$ is chosen appropriately. Applying Theorem \ref{t:meta} gives
	\[
	\sup_{U \cap B_{c_*/2}} \frac{u}{v} \leq C\sup_{B_{c_*/2} \cap U} \frac{u_1}{u_2} \leq C.
	\]
	
	The statement as written (on $U \cap B_{1/2}$) then follows after a standard covering argument.
\end{proof}

 \section{p-Laplacian boundary Harnack}  
 
We now demonstrate the versatility of this approach by showing the same result for the $p$-Laplacian 
$(1<p< \infty)$ defined through $\hbox{div} (|\nabla u |^{p-2} \nabla u) $.
 The analogue of Lemma \ref{l:classic2} (the boundary Harnack principle for the $p$-Laplace with right hand side zero) is proven in \cite{ln10a}. We will also need the analogues of Lemma \ref{l:remainder} and Lemma \ref{l:growthnonlinear} for the $p$-Laplacian. The analogue of Lemma \ref{l:growthnonlinear} is proven in the same manner for the $p$-Laplacian. However, proving Lemma \ref{l:remainder} for the $p$-Laplacian is more difficult: to see why, note that a difference $u - v$ of two solutions to the (inhomogeneous) $p$-Laplacian does not satisfy a PDE of the same type; rather, at best it satisfies a kind of linearized equation with coefficients dependent on $\nabla u$ and $\nabla v$. Our approach here will be to establish bounds on these gradients and then work with this linearized equation.

\subsection{Growth estimates}

\begin{lemma}\label{l:growthplaplace}  Let $\beta \in (1, 2)$. Then there are constants $\epsilon, \eta > 0$ such that if $u \geq 0 $ on $D_{L, 1}$, $u(e_n/2) \geq 1$, $u = 0$ on $\partial D_{L, 1}\cap B_1$,
	\[
	- 1 \leq \Delta_p u \leq \epsilon,
	\]
	and $L \leq \eta$, then
	\[
	u(x) \geq c_1 d^\beta(x, \partial D_{L, 1})
	\]
	for $x\in D_{L, 1/64}$. 
\end{lemma}

\begin{proof}
The proof follows that of Lemma \ref{l:growthnonlinear} with minor modifications, which we explain here. First, in the proof of Lemma \ref{l:barrier} we used a barrier function $\phi(x) = |x|^{-q}$ which was a subsolution to the equation on $B_1 \setminus \{0\}$. For large values of $q$ (depending on $p$), this also has $\Delta_p u \geq 1$. Indeed, the $p$-Laplacian is given by
\[
\Delta_p \phi = |\nabla u|^{p - 2}Tr [I + (p - 2) \frac{\nabla \phi \otimes \nabla \phi}{|\nabla \phi|^2} ] D^2 \phi.
\]
The matrix in square brackets is independent of the form of $\phi$ for any radial, radially decreasing function, and at the point $r e_n$ the $ij$-th entry is given by $\delta_{ij} + \delta_{in}\delta_{jn}(p - 2)$. Computing $D^2\phi$ at this point, one may check that this is a diagonal matrix in the $e_i$ basis with $\partial_{ii} \phi =  - q r^{-q -2}$ for $i < n$ and $\partial_{nn} \phi = q (q - 2) r^{- q - 2}$. This gives that at $r e_n$,
\[
 Tr [I + (p - 2) \frac{\nabla \phi \otimes \nabla \phi}{|\nabla \phi|^2} ] D^2 \phi \geq c(p, n) q (q - 2) r^{-q - 2}
\]
for all $q$ sufficiently large enough. On the other hand, $|\nabla \phi| = q r^{- q - 1}$, so
\[
\Delta_p \phi \geq c q ( q - 2) r^{- q - 2} [ q r^{- q - 1} ]^{p - 2 }  \geq 
  c q^{p - 1} ( q - 2) r^{- [q + 2 + (p - 2)( q + 1)]}.
\]
As $p > 1$, the exponent in square brackets $q + 2 + (p - 2)( q + 1) > 1$, so
\[
	\Delta_p \phi \geq c q^{p - 1}(q - 2) r^{-1} \geq 1
\]
so long as $q$ is chosen large enough in terms of $c$ and $p$.

The only other modification needed is in \eqref{e:growth} in the proof of Lemma \ref{l:growthnonlinear}, where we rescale $u_k(y) = u(x + r_k y)/u(x + r_k e_n/2)$ and compute the PDE. Here our equation is different, but we still have that
\[
\Delta_p u_k  (x) = \frac{r_k^p}{u^{p - 1}(x + r_k e_n/2)} (\Delta_p u  ) (x + r_k y ),
\]
which implies $- 1 \leq \Delta_p u_k \leq \epsilon$ 
so long as $u(x + r_k e_n/2) \geq r_k^{\frac{p}{p-1}}$. This may be ensured by replacing the condition 
$c_*/8 > 64 \eta$ with $c_*/8 > (64 \eta)^{\frac{1}{p - 1}}$ on $\eta$.
\end{proof}

An elementary Harnack principle at the boundary:
\begin{lemma}\label{l:bdrybdd} Assume $L \leq \frac{1}{100}$. Let $u\geq 0$ on $D_{L, 1}$ and satisfy $- 1\leq  \Delta_p u \leq 1$. Assume, moreover, that $u = 0$ on $\partial D_{L, 1} \cap B_1$, the graph part of the boundary, and $u(e_n/2) = 1$. Then there is a constant $C = C(n, p)$ such that
	\[
	\sup_{D_{L, 1/2}} u \leq C.
	\]
\end{lemma}

\begin{proof}
	We argue by contradiction, assuming that the conclusion is false. First, observe that from the Harnack inequality, we have that $ u \leq C [u(e_n/2) + 1]\leq C$ on $B_{1/2} \cap \{x_n \geq \frac{1}{10}\}$. However, as the conclusion of the lemma is false (with this particular $C > 1$), we may then find a point $x^1 \in D_{L, 1/2} \cap \{x_n \leq \frac{1}{10} \}$ with $u(x^1) > C$. Let $z^1$ be the unique point in $\partial D_{L, 1} \cap B_{1}$ with $(z^1)_n = (x^1)_n$, and set $u_1(y) = u(z^1 + r_1 y)/u(x^1)$, where $r_1 = 2 |z^1 - x^1|$. Let $D^1_{L, 1} = (D_{L, 1} - z^1)/r_1 \subset B_{r_1}(z^1) \subset B_{4/10}(x^1) \subset B_{9/10}$, noting this is still a Lipschitz set with the same constant $L$. Then $u_1 \geq 0$ on $D^1_{L, 1}$, has $u_1(e_n/2) = 1$, and satisfies $- r^p/C^{p - 1} \leq \Delta_p u_1 \leq r^p/C^{p - 1}$ on it.
	
	Now we repeat the process, finding a point $w$ for $u_1$, in the same region as $x^1$ for $u$, where $u_1(w) > C$. Set $x^2 = z^1 + r w \in B_{9/10}$, and continue the process in this way, constructing a sequence of points $\{x^k\} \subset B_{9/10}$, where $u(x^k) \geq C^k$. 
	After a finite number of steps, we will find a point where 
	$u(x^k) \geq \max_{D_{L, 9/10}} u$, which is finite as $u$ is continuous.  
	This is a contradiction.
\end{proof}

There is an analogous version of the lower bound lemma which gives bounds from above instead. The proof is identical apart from using barriers from above, which may be constructed from the functions $C - \phi$.

\begin{lemma}\label{l:upperbdplaplace} Let $\beta \in (0, 1)$. Then there is a constant $\eta > 0$ such that if $u \geq 0 $ on $D_{L, 1}$, $u(e_n/2) = 1$, $u = 0$ on $\partial D_{L, 1}\cap B_1$,
	\[
	- 1 \leq\Delta_p u \leq 1,
	\]
	and $L \leq \eta$, then
	\[
	u(x) \leq C_1 d^\beta(x, \partial D_{L, 1}\cap B_1)
	\]
	for $x\in D_{L, 1/64}$.
\end{lemma}

\begin{proof} 

	For the first conclusion, the proof proceeds similarly to the proof of Lemma \ref{l:growthplaplace}, with two minor differences: first, apply Lemma \ref{l:bdrybdd} to ensure that $u \leq C$ in $D_{L, 1/2}$, after which the analogue of Lemma \ref{l:barrier} follows easily by using $C - \phi$ in place of $\phi$ as a barrier (with the inner part of the annulus outside the domain). When proving the analogue of Lemma \ref{l:growthnonlinear}, first apply Lemma \ref{l:growthplaplace} to give that
	\[
		u(x) \geq c d^{\beta_2}(x, \partial D_{L, 1}\cap B_1) \geq d^{\frac{p}{p - 1}}(x, \partial D_{L, 1}\cap B_1)
	\]
	for $|x|$ small enough, choosing $\beta_2$ close to $1$. This alone guarantees that all of the functions $u_r(z) = u(y + r z)/u(y + r e_n/2)$ (for $y \in \partial D_{L, 1}\cap B_1$ and $r$ small) satisfy the same differential inequalities as $u$:
	\[
		-1 \leq \Delta_p u_r \leq 1.
	\] 
	The proof then proceeds similarly to the lower bound.
\end{proof}

\subsection{Derivative lower bounds}

We use the following Liouville-type result. 

\begin{lemma}  \label{l:lious}
	If $u,v \geq 0$ and $u,v$ satisfy $\Delta_p u, \Delta_p v = 0$ on $D_{L, \infty}$, with $u, v = 0$ on $\partial D_{L, \infty}$, then $u=cv$ for some constant $c$. 
\end{lemma}

\begin{proof}
	We assume that $u,v$ are both not identically zero. We use the classical boundary Harnack principle from \cite{ln10} to show that $u/v$ is uniformly H\"older continuous up to the boundary on any compact subset of $\mathbb{R}^n$.  We normalize $u$ so that $\lim_{x \rightarrow 0} u(x)/v(x)=1$, and let $x^0 \in D_{L,\infty}$. The rescaled functions 
	\[
	u_R(x):= \frac{u(Rx)}{u(Re_n)}    \qquad v_{R}:=\frac{v(Rx)}{v(Re_n)}
	\]
	are also $p$-harmonic. Furthermore, we have that $\lim_{x \rightarrow 0} u_R(x)/v_R(x)=1$,  and that $u_R/v_R$ is H\"older continuous on $D_{L,2}$ with norms independent of $R$. 
	Then by continuity of the quotient, for any $\epsilon>0$ there exists $R$ large enough, so that 
	\[
	\left| \frac{u_R(x^0/R)}{v_R(x^0/R)} - \frac{u_R(0)}{v_R(0)}\right| \leq C |x^0|/R < \epsilon. 
	\]
	We note that 
	\[
	\frac{u_R(x^0/R)}{v_R(x^0/R)} = \frac{u(x^0)}{v(x^0)},
	\]
	and therefore  $u(x^0)/v(x^0) = 1$, implying $u = v$.
\end{proof}

As a consequence, on $D_{L, \infty} = \R_+^n$ we have $u = c (x_n)_+$ in the configuration above. We exploit this fact below.

\begin{theorem}  \label{l:gradbddbelow}
	There exist constants $C, \epsilon, \eta, r > 0$ depending only on $n$ and $p$ such that if $L < \eta$,
	\[
	\begin{cases}
	-1 \leq \Delta_p u \leq \epsilon & \text{ on } D_{L,2} \\
	u = 0 \text{ on } & \partial D_{L,2} \cap B_{2} \\
	u(e_n)=1 &\\
	u \geq 0 & \text{ in } D_{L,2}
	\end{cases}
	\]
	then 
	\[
	\frac{1}{C} \frac{u(x)}{d(x,\partial D_{L,2})} \leq |\nabla u(x)| \leq C  \frac{u(x)}{d(x,\partial D_{L,2})} \text{ whenever }
	x \in B_{r}.
	\]
\end{theorem}

\begin{proof}
	Suppose by way of contradiction that the theorem is not true. Then there exists $u_k, D^k_{L_k, 2}, \epsilon_k$ satisfying the assumptions with $\epsilon_k \to 0$
	and $x^k \in B_{r_k}$ satisfying such that either
	\begin{equation}  \label{e:updown}
	|\nabla u_k (x^k)| \leq \frac{1}{C} \frac{u_k(x^k)}{d(x^k,\partial D_{L_k,2}^k)} \quad \text{or} \quad
	|\nabla u_k (x^k)| \geq C\frac{u_k(x^k)}{d(x^k,\partial D_{L_k,2}^k)}.
	\end{equation}
	
	Apply Lemma \ref{l:growthplaplace} to $u_k$ to obtain (for some $c>0$)
	\[
	c d^{\beta_1}(z, \partial D_{L_k, 2}^k) \leq u_k(z)
	\]
	on $B_{1/32}$ for a $\beta_1 > 1$ to be chosen. Set $y^k = (x'_k, g_k(x'_k)) \in \partial D^k_{L_k, 2}$, the projection of $x^k$ onto the graphical part of the boundary of $D^k_{L_k, 2}$ and $s_k = |x^k - y^k| \leq r_k$. We rescale with 
	\[
	\tilde{u}_k(x) := \frac{u_k(y^k + 2 s_k x)}{u_k(y^k + s_k e_n)}.
	\]
	Note that $y^k + s_k e_n = x^k$.
	
	Let us verify the differential inequalities satisfied by $\tilde{u}_k$: if $A_k = \frac{(2s_k)^p}{u_k^{p - 1}(y^k + s_k e_n)}$, then
	\[
		- A_k \leq \Delta_p \tilde{u}_k \leq A_k \epsilon_k
	\]
	on $\tilde{D}^k_{L_k, 1/s_k} = (D^k_{L_k, 2} - y^k)/s_k \cap B_{1/s_k}$. We claim that $A_k \rightarrow 0$. Indeed,
	\[
		A_k \leq \frac{(2s_k)^p}{c s_k^{\beta_1 (p - 1)}} \leq C r_k^{p - \beta_1 (p - 1)},
	\]
	which converges to $0$ so long as $\beta_1 < p/(p-1)$. 
	
	Next, fix any large $R$.  We have $\tilde{u}_k\geq 0$ and $\tilde{u}_k(e_n) = 1$ by construction. Applying Lemma \ref{l:upperbdplaplace} a finite number of times to $\tilde{u}_k$ on progressively larger balls which exhaust $\tilde{D}^k_{L_k, R}$, we obtain that
	\begin{equation}\label{e:upperbd}
		\tilde{u}_k(z) \leq C(R) d^{\beta_2}(z, \partial \tilde{D}_{L_k, 2}^k)
	\end{equation}
	for a fixed $\beta_2 < 1$. Meanwhile, on any $U = B_R \cap \{ x_n > \delta \}$, which lies entirely inside $\tilde{D}^k_{L_k, 1/s_k}$ for $k$ large, from standard interior $C^{1, \alpha}$ estimates we have that
	\[
		\|\tilde{u}_k\|_{C^{1, \alpha}(U)} \leq C.
	\]
	We may extract a subsequence along which $\tilde{u}_k$ converges in $C^{1, \alpha}(U)$ for every set $U$ to a limiting function $u \geq 0$ on $\R^n_+$. From \eqref{e:upperbd}, we have that $u$ is continuous up to $\{x_n = 0\}$ and vanishes along that set. The PDE passes to the limit to give that $\Delta_p u = 0$. We also have $u(e_n) = 1$, and
	\[
		 \frac{\nabla u_k(x^k)}{u_k(x^k)} = \nabla \tilde{u}_k(e_n) \rightarrow \nabla u(e_n),
	\]
	meaning that $|\nabla u(e_n)| \notin [\frac{1}{C}, C]$. Applying Lemma \ref{l:lious}, however, gives that $u(x) = x_n$ (using $u(e_n) = 1$ here), so this is impossible.
\end{proof}

\subsection{The approximation lemma}

\begin{lemma}\label{l:degiorgi} Let $u$ be an $H^1_0$ function on $D_{L, 1}$, $L \leq \frac{1}{10}$, which satisfies
	\[
		\int_{D_{L, 1}} A \nabla u \cdot \nabla \phi \leq \int_{D_{L, 1}} \phi
	\]
	for all nonnegative $\phi \in C^1_c(D_{L, 1})$, where $A$ is a measurable matrix-valued function with
	\[
		\lambda d^\epsilon(x, \partial D_{L, 1}\cap B_1) I \leq A \leq \lambda^{-1} d^{-\epsilon}(x, \partial D_{L, 1}\cap B_1) I.
	\]
	Then if $\epsilon < \epsilon_0$ small enough, we have that
	\[
		\sup_{D_{L, 1}} u \leq C(\lambda).
	\]
\end{lemma}

\begin{proof}
	Let $u_k = (u - l_k)_+$, where $\{l_k\} $ is a strictly increasing sequence of real numbers. A straightforward approximation argument shows that $u_k$ may be used as test functions, i.e.
	\begin{equation} \label{e:degiorgi}
		\int \lambda d^\epsilon |\nabla u_k|^2 \leq \int A \nabla u \cdot \nabla u_k \leq \int u_k
	\end{equation}
	where $d(x) = d(x, \partial D_{L, 1}\cap B_1)$. Applying the H\"older inequality,
	\[
		\int |\nabla u_k|^{2 \alpha} = \int |\nabla u_k|^{2 \alpha} \frac{d^{\epsilon\alpha}}{d^{\epsilon\alpha}} \leq (\int d^{\epsilon} |\nabla u_k|^2)^\alpha (\int d^{\epsilon\alpha ( 1 - \alpha)}(x))^{\frac{1}{1 - \alpha}}.
	\]
	For any $\alpha < 1$, we may choose $\epsilon$ small enough that the rightmost factor is bounded. Now from the Sobolev embedding,
	\[
		\|u_k\|_{L^{\frac{2 \alpha n}{n - 2 \alpha}}} \leq C \|\nabla u_k\|_{L^{2\alpha}} \leq C (\int d^\epsilon|\nabla u_k|^2)^{\frac{1}{2}}.
	\]
	Choose $\alpha < 1$ so the exponent $q := \frac{2 \alpha n}{n - 2 \alpha} > 2$.
	 Applying  \eqref{e:degiorgi} to the right and raising to the $q$-th power,
	\[
		\int u_k^q \leq C (\int d^\epsilon|\nabla u_k|^2)^{q/2} \leq C (\int u_k)^{q/2}. 
	\]
	In particular, applying H\"older's inequality to the right and dividing gives
	\begin{equation}\label{e:degiorgi2}
		\int u_k^q \leq C.
	\end{equation}
		Alternatively, we can obtain the recursion formula
	\begin{equation}\label{e:degiorgi3}
		\int u_{k + 1} \leq \int_{u_k > l_{k+1} - l_{k}} u_k \leq \frac{1}{(l_k - l_{k-1})^{q - 1}}\int u_k^q \leq \frac{C}{(l_k - l_{k-1})^{q - 1}} (\int u_k)^{q/2}.
	\end{equation}

	Next, select $l_k = 2^k$: for any $K > 2$, choosing $m$ so that $2^{m} \leq K \leq 2^{m + 1}$ and then combining \eqref{e:degiorgi2} and \eqref{e:degiorgi3} gives
	\begin{equation}\label{e:degiorgi4}
		\int (u - K)_+ \leq \int ( u - 2^{m} )_+ = \int u_m \leq \frac{C}{2^{(m - 1)(q - 1)}} (\int u_{m - 1})^{q/2} \leq \frac{C}{K^{q - 1}}.
	\end{equation}
	
	Now we make a different selection of $l_k$: $l_k = K + 1 - 2^{-k}$ with $K > 2$ large. From \eqref{e:degiorgi3},
	\[
		\int u_{k+1} \leq C 2^{k(q - 1)} (\int u_k)^{q/2}.
	\]
	If $\int u_0 \leq \delta$ for some $\delta$ depending on $C$ and $q$ here, the sequence $\{ \int u_k \}_{k = 1}^\infty$ converges to $0$, which would give that $u \leq K + 1$. Using \eqref{e:degiorgi4}, though,
	\[
		\int u_0 = \int (u - K)_+ \leq \frac{C}{K^{q - 1}} \leq \delta
	\]
	if $K$ is chosen large enough in terms of $C$ and $q$. Thus for a large enough $K$, $u\leq K + 1$, which implies the conclusion.
\end{proof}

\begin{lemma} \label{l:plaplaceapprox}
	There exist constants $\eta, \epsilon, r$ small such that if $L \leq \eta$, $u \geq 0$ on $D_{L, 1}$, $u = 0$ on $\partial D_{L, 1} \cap B_1$, $u(e_1/2) = 1$, and $-A_0 \leq\Delta_p u \leq A_0 \epsilon$ on $D_{L, 1}$ for some $A_0 \leq 1$, the following holds: if $w$ satisfies
	\[
 	\begin{cases}
	 	\Delta_p w = 0 & \text{ on } D_{L, r} \\
	 	w = u & \text{ on } \partial D_{L, r},
 	\end{cases}
	\]
	then $| w - u|\leq C A_0$.
\end{lemma}

\begin{proof}
	Set $d(x) = d(x, \partial D_{L, 1} \cap B_1)$ below. Let $f$ solve the following PDE:
	\[
	\begin{cases}
	\Delta_p f = -1 & \text{ on } D_{L, 1} \\
	f = u & \text{ on } \partial D_{L, 1}.
	\end{cases}	
	\]
	From the maximum principle, $f \geq u$ and $f \geq w$. In particular $C \geq f(e_1/2) \geq 1$, with the upper bound from the Harnack inequality, so applying Lemma \ref{l:upperbdplaplace} to $f/f(e_1/2)$ gives that
	\[
	u(x), w(x) \leq f(x) \leq C d^{\beta_1}(x)
	\]
	on $D_{L, 1/64}$ for $\beta_1 < 1 $ fixed.
	
	Apply Lemmas \ref{l:growthplaplace}  and \ref{l:gradbddbelow} to $u$ for $\beta_2 > 1$ fixed, choosing $\eta$ and $\epsilon$ so the assumptions are satisfied regardless of $A_0$. Set $r$ to the smaller of the $r$ in Lemma \ref{l:gradbddbelow} and $1/64$; then we have that
	\[
	c d^{\beta_2}(x) \leq  u(x)
	\]
	and
	\[
	|\nabla u(x)| \approx \frac{u(x)}{d(x)},
	\]
	so
	\[
	c d^{\beta_2 - 1}(x) \leq |\nabla u(x)| \leq C d^{\beta_1 - 1}(x)
	\]
	for $x \in D_{L, r}$. Now, take any $x \in D_{L, r}$ and $B_{d(x)/2}(x)$: on this ball, we may apply either the boundary or interior form of the $C^{1, \alpha}$ estimate for $p$-harmonic functions \cite{gt01} to give that
	\[
	|\nabla w(x)|\leq C \frac{\max\{  d^{\frac{p}{p - 1} }(x), \sup_{B_{d(x)/2} } w  \}}{d(x)} \leq C d^{\beta_1 - 1}(x).
	\]
	Next, set
	\[
		a(x) = \int_0^1 |\nabla u(x) t + \nabla w(x) (1 - t)|^{p - 2} dt.
	\]
	The quantities $\nabla u, \nabla w$ are locally bounded on the set $D_{L, r}$, so when $ p \geq 2$ this is well-defined on this region. When $p < 2$, note that $\nabla u \neq 0$, and so the integrand is an integrable function regardless of the value of $\nabla w$, meaning $a$ is still well-defined. In a similar vein, we estimate $a$ from above and below. If $p \geq 2$, then
	\[
		a(x) \leq C[|\nabla u(x)|^{p - 2} + |\nabla w(x)|^{p - 2} ] \leq C d^{(\beta_1 - 1)(p - 2)} \leq C d^{- \alpha}
	\]
	so long as $\beta_1$ is chosen large enough relative to $\alpha$, which will be determined below. When $p < 2$, the same computation instead gives
	\[
		a(x) \geq [|\nabla u(x)| + |\nabla w(x)| ]^{p - 2} \geq c d^{(\beta_1 - 1)(p - 2)} \geq c d^{\alpha}.
	\]

	On the other hand, we have 
	\[
		|\nabla u(x) t + \nabla w(x) (1 - t)| \geq t |\nabla u| - (1 - t)|\nabla w| \geq \frac{1}{4} |\nabla u|
	\] for $ t\geq \frac{3}{4}$ if $|\nabla w| \leq |\nabla u|$. If instead $|\nabla w|\geq |\nabla u|$, we get
	\[
		|\nabla u(x) t + \nabla w(x) (1 - t)| \geq (1 - t) |\nabla w| - t|\nabla u| \geq \frac{1}{4} |\nabla w| \geq \frac{1}{4} |\nabla u|
	\]
	for $t < \frac{1}{4}$. In either case this holds on an interval of length $\frac{1}{4}$, so if $p > 2$,
	\[
		a(x) \geq c |\nabla u|^{p - 2} \geq c d^{(\beta_2 - 1)(p - 2)}(x) \geq c d^{\alpha}(x)
	\]
	if $\beta_2$ is small enough. Finally, for $p < 2$ one may check that
	\[
		\int_0^1 |\nabla u(x) t + \nabla w(x) (1 - t)|^{p - 2}  dt \leq C |\nabla u|^{p - 2} \leq C d^{(\beta_2 - 1)(p - 2)} \leq C d^{-\alpha}
	\]
	by directly computing the integral. To summarize, we have shown that
	\[
		c d^{\alpha} \leq a(x) \leq C d^{ -\alpha}.
	\]

	Consider now the matrix
	\[
		a_{ij}(x) = \int_0^1 |\nabla u(x) t + \nabla w(x) (1 - t)|^{p - 2} m_{ij}^t  dt.
	\]
	where
	\[
		m_{ij}^t = \delta_{ij} + (p - 2) \frac{(u_i t + w_i (1 - t))(u_j t + w_j (1 - t))}{|\nabla u(x) t + \nabla w(x) (1 - t)|^2}.
	\]
	For any fixed $t$ and $\xi \in \R^n$, the factor $m_{ij}^t$ has $\lambda |\xi|^2 \leq m_{ij}^t \xi_i\xi_j \leq \lambda^{-1} |\xi|^2$, with constant $\lambda$ depending only on $p$. Using this,
	\[
		a_{ij}(x) \xi_i \xi_j = \int_0^1 |\nabla u(x) t + \nabla w(x) (1 - t)|^{p - 2} m_{ij}^t \xi_i \xi_j dt \geq a(x) \lambda |\xi|^2,
	\]
	and similarly $a_{ij}(x) \leq \lambda^{-1} |\xi|^2 a(x)$.
	
	The point of this $a_{ij}$ is that, if $F(z) = |z|^{p - 2} z$,
	\[
		F_i(\nabla u) - F_i(\nabla w) = \int_0^1  \partial_t F(\nabla u(x) t + \nabla w(x) (1 - t))  dt = a_{ij}(x) (u_j - w_j).
	\]
	Setting $h = u - v$, we have shown that
	\[
		\Delta_p u -\Delta_p w = \text{div} [F(\nabla u)  - F(\nabla w)] = \partial_i (a_{ij}  h_j)
	\]
	on $D_{L, r}$ (in the distributional sense). In particular,
	\[
		-A_0\leq \partial_i (a_{ij}  h_j) \leq A_0
	\]
	from the equations on $u$ and $w$. Apply Lemma \ref{l:degiorgi} to $\pm \frac{h(r \cdot)}{A_0}$, using our bounds on $a_{ij}$ and choosing $\alpha$ small enough, to get that
	\[
		|h| \leq C A_0
	\]
	on $D_{L, 1}$. This completes the argument.
\end{proof}

We may reformulate this approximation lemma in a more helpful way:

\begin{lemma}\label{l:plaplaceapproxscaled}
	For every $\alpha > 0$, there exist constants $\eta, \epsilon, r_0$ small such that if $L \leq \eta$, $u \geq 0$ on $D_{L, 1}$, $u = 0$ on $\partial D_{L, 1} \cap B_1$, $u(e_1/2) = 1$, and $-1 \leq\Delta_p u \leq \epsilon$ on $D_{L, 1}$, the following holds: if $w$ satisfies
	\[
	\begin{cases}
	\Delta_p w = 0 & \text{ on } D_{L, r} \\
	w = u & \text{ on } \partial D_{L, r},
	\end{cases}
	\]
	with $ r \leq r_0$, then $|w - u|\leq C r^{2 - \alpha}$.
\end{lemma}

\begin{proof}
	First, apply Lemma \ref{l:growthplaplace} to $u$ to obtain that $c d^{\beta}(x) \leq  u(x)$ for a $\beta$ to be determined shortly on $D_{L, r_0}$. Set
	\[
		u_1(y) = \frac{u(s y)}{u(s e_n/2)},
	\]
	where $s = \frac{r}{r_1}$, where we set $r_1$ to be the $r$ in Lemma \ref{l:plaplaceapprox}'s conclusion and ask that $r_0 \leq r_1^2$. Set $w_1(y) = \frac{w(s y)}{u(s e_n/2)}$.
	Let us check the equation satisfied by $u_1$ on $D'_{L, 1}$ (the rescaled domain):
	\[
		\Delta_p u_1(y) = \frac{s^p}{u^{p - 1}(s e_n/2)}\Delta u(y/s) := A\Delta u(y/s).
	\]
	We wish to arrange to have $A_0 \leq 1$. This may be done, as
	\[
		u^{p - 1}(s e_n/2) \geq c s^{(p - 1) \beta} \geq s^{p - \frac{1}{2}},
	\]
	where we choose $\beta$ sufficiently close to $1$, and then $r_0$ small enough so as to have $s \leq r_0/r_1$ absorb the constant.
	Apply Lemma \ref{l:plaplaceapprox} to deduce that
	\[
		|u_1 - w_1|\leq C A_0,
	\]
	which scales back to
	\[
		|u - w|\leq C A_0 u(s e_n/2) \leq C \frac{s^p}{u^{p - 2}(s e_n/2)}.
	\]
	As before, we may estimate
	\[
		u^{p - 2}(s e_n/2) \geq c s^{(p - 2) \beta} \geq c s^{p - 2 + \alpha},
	\]
	by choosing $\beta$ close to $1$, so that
	\[
		|u - w| \leq C s^{2 - \alpha}  \leq C r^{2 - \alpha}.
	\]
\end{proof}

\begin{proof}[Proof of Theorem \ref{t:flpmain}.]
	We apply Theorem \ref{t:meta} with $H$ the solution mapping for the $p$-Laplacian, $U$ our Lipschitz graph domain $D_{L, 1}$, $U_{x, r} = U \cap B_r(x)$, and $V$ the set of all $u$ with $u > 0$ on $U$, $u = 0$ on $\partial U$, $u(e_n/2) = 1$, and $-1 \leq \Delta_p u \leq \epsilon$ for $\epsilon$ small. Then all of the properties (P1-5) and (P7-8) follow in a standard way. For property (P6), we apply Lemma \ref{l:plaplaceapproxscaled} to $u \in V$ to see that at least it is valid when centered at $x = 0$ and $r < r_0$. For $ r \geq 0$, the property is automatic from the bound in Lemma \ref{l:bdrybdd} instead. For other $x \in \partial U \cap B_{1/2}$, it then follows from a simple translation argument.
	
	Lemma \ref{l:growthplaplace} ensures that the growth assumptions on $u, v$ hold, so we may apply Theorem \ref{t:meta} to $u, v\in V$. The rest follows as in the proof of Theorem \ref{t:nta} or \ref{t:flmain}.
\end{proof}
  
\section*{Acknowledgments}

HS was supported by Swedish Research Council.

\bibliographystyle{plain}
\bibliography{refnonlinear}

\end{document}